\documentclass[review]{elsarticle}

\usepackage{amssymb}
\usepackage{amsthm}
\usepackage[utf8]{inputenc} 
\usepackage[T1]{fontenc}    
\usepackage{hyperref}       
\usepackage{url}            
\usepackage{booktabs}       
\usepackage{amsfonts}       
\usepackage{nicefrac}       
\usepackage{microtype}      
\usepackage{lipsum}
\usepackage{amsmath}
\usepackage{amsthm}
\usepackage{subfigure}
\usepackage[left=2.5cm,right=2.5cm,top=2.5cm,bottom=2.5cm]{geometry}

\newtheorem{lemma}{Lemma}
\newtheorem{proposition}{Proposition}

\theoremstyle{definition}

\newtheorem{example}{Example}

\journal{\hspace{1cm}}

\begin{document}

\begin{frontmatter}

\title{Analysis of infected population threshold exceedance in an SIR epidemiological model}


\author[1]{Andrés David Báez Sánchez\corref{cor1}}
\ead{adsanchez@utfpr.edu.br}
\author[1]{Nara Bobko}
\ead{narabobko@utfpr.edu.br}
\cortext[cor1]{Corresponding author}

\address[1]{Department of Mathematics, Federal University of Technology, Av. Sete de Setembro, 3165, 80230-901, Curitiba, Paraná, Brazil.}

\begin{abstract}
   We consider an epidemiological SIR model and a positive threshold $M$. Using a parametric expression for the solution curve of the SIR model and the Lambert W function, we establish necessary and sufficient conditions on the basic reproduction number $\mathcal{R}_0$ to ensure that the infected population does not exceed the threshold $M$. We also propose and analyze different measures to quantify a possible threshold exceedance.
\end{abstract}


\end{frontmatter}

\section{Introduction}
\label{sec.intro}

In the context of the recent COVID-19 outbreak, much attention has been given to the idea of \textit{flattening the curve} of infection (Figure~\ref{Figure0}), to reduce the harmful effects of an overloaded medical system~\cite{TWP,CBS,DMUK}. When the number of active cases of a given infectious disease is over the healthcare capacity, it is expected to observe a rise in the disease mortality rate, as well as an increase in other diseases mortality, because of a general decrease in the medical care quality.

\begin{figure}[!ht]
    \center
    \includegraphics[width=0.7\textwidth]{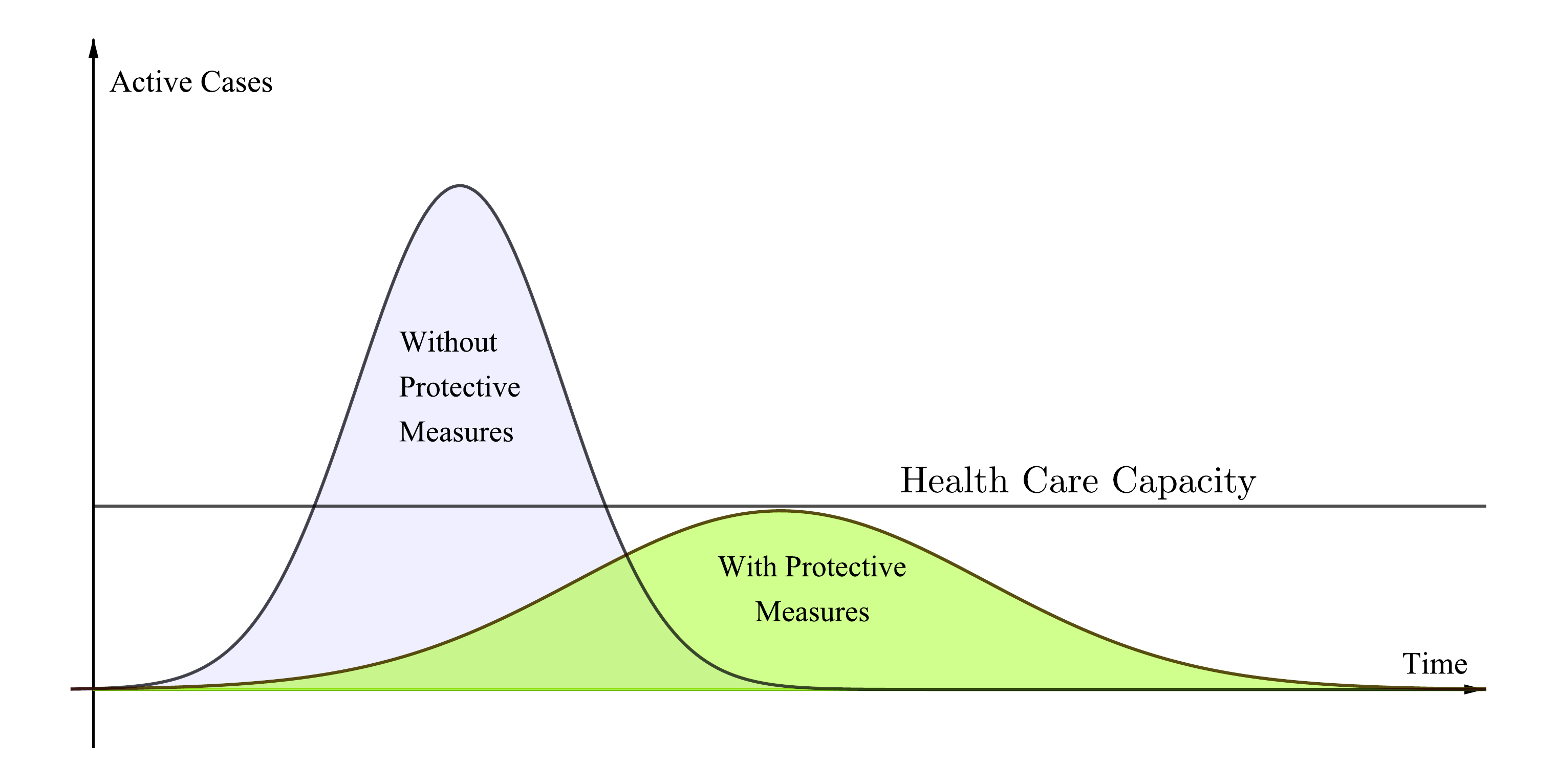}
    \caption{Flattening the curve.}
    \label{Figure0}
\end{figure}

It is desirable to determine specific conditions to ensure that the health care capacity will not be exceeded, and if such a situation is inevitable, it would be useful to  effectively quantify the negative impact produced. 

In this work, we consider an epidemiological SIR model and a positive threshold $M$. We establish necessary and sufficient conditions on the reproduction number $\mathcal{R}_0$ to ensure that the infected population does not exceed the threshold $M$, and also propose and discuss five different measures to quantify the impact of such exceeding.

If one considers that at a given instance, the negative impact is proportional to the current level of threshold exceeding, then, the area under the infected curve and over the threshold may be considered as a possible measure of the total negative impact on the healthcare system (Figure~\ref{Figure1}).

\begin{figure}[!ht]
    \center
    \includegraphics[width=0.7\textwidth]{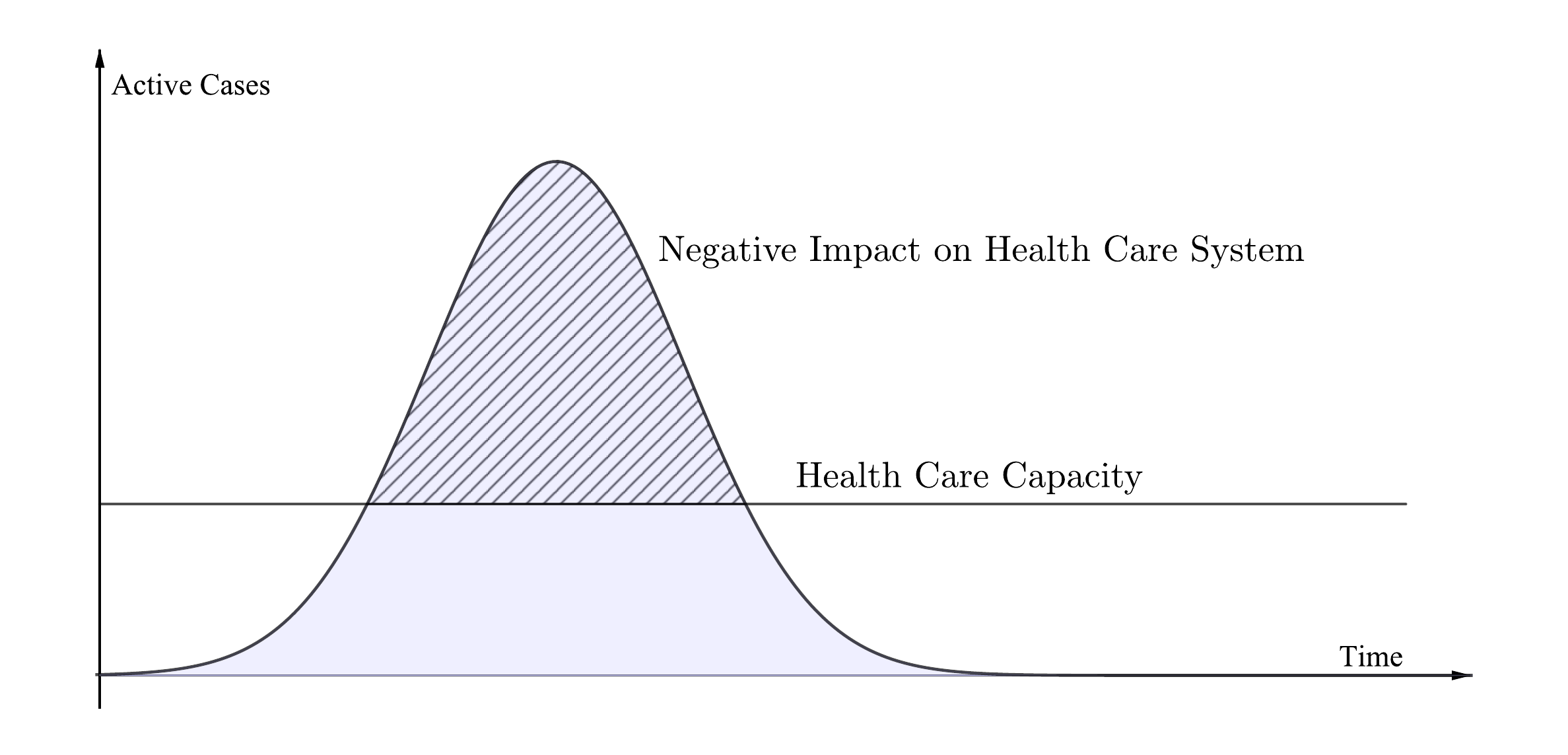}
    \caption{Impact on Health Care System.}
    \label{Figure1}
\end{figure}

This kind of measure will be considered, as well as four other quantifiers that we propose as possible measures for threshold exceeding quantification.

Along this paper, we will use an exact parametric solution for the SIR model, based on the work of Harko \textit{et.al.} in~\cite{HARKO2014184}, and will also use the Lambert W function, also known as the product logarithm function~\cite{Corless1996}. Previous works have considered the Lambert W function in the context of epidemiological models. In~\cite{RELUGA2004249,Wang2010,Pakes2015} for example, the Lambert W function is used to express the final sizes in different epidemiological models and in~\cite{Xiao2013}, the Lambert W function is used to manipulate and analyze an epidemiological model with piece-wise smooth incidence rate. Additional details and applications of the Lambert W function can be found in~\cite{Corless1996,Lehtonen2016}.

The rest of this work is organized as follows: Section~\ref{sec.model} briefly recall the epidemiological SIR model and present a parametric solution according to the results in~\cite{HARKO2014184}. Section~\ref{sec.Imax} considers the control of the maximum value of the infected population. We establish conditions on the reproduction number to ensure that the maximum infected value does not exceed $M$. In Section~\ref{sec.quantifying} we propose and analyze five different measures to quantify the impact caused by a possible threshold exceeding. Final considerations are presented in the Section~\ref{sec.finalcomments}. Appendix A, includes the main results related to the Lambert W function used along the work.

\section{An epidemiological SIR model an its parametric solution} \label{sec.model}

The main idea behind SIR models is to consider that a population $N$ is divided into three disjoint categories or compartments: susceptible individuals, infected individuals, and removed individual (recovered or deceased individuals), denoted by $S$, $I$ and $R$ respectively, so $N=S+I+R$.

We consider the following SIR model,
\begin{equation}
    \begin{split}
    \dfrac{dS}{dt} &= -\beta\,S\,I\\
    \dfrac{dI}{dt} &= \beta\,S\,I-\gamma I\\
    \dfrac{dR}{dt} &= \gamma \, I.\\
    \end{split}
    \label{model1} 
\end{equation}

The positive real numbers $\beta$, and $\gamma$ are interpreted as the infection rate and recovery rate respectively. Note that adding equations in~\eqref{model1}, we can obtain $$\dfrac{dN}{dt} = \dfrac{dS}{dt} + \dfrac{dI}{dt} + \dfrac{dR}{dt} = 0.$$
Thus, we can consider  $N(t)=S(t) + I(t) + R(t)$ constant for all $t$. Letting $\mathcal{R}_0=\frac{\beta N}{\gamma}$, we can rewrite~\eqref{model1} as

\begin{equation}
    \begin{split}
    \dfrac{dS}{dt} &= -\frac{\gamma\,\mathcal{R}_0\,S\,I}{N}\\
    \dfrac{dI}{dt} &= \frac{\gamma\,\mathcal{R}_0\,S\,I}{N}-\gamma I\\
    \dfrac{dR}{dt} &= \gamma \, I.\\
    \end{split}
    \label{model2} 
\end{equation}

The parameter $\mathcal{R}_0$ is called the basic reproduction number and has fundamental role in the description of the equilibria stability in the classical SIR model~\cite{martcheva2015introduction}. 
The reproduction number $\mathcal{R}_0$ can be interpreted as the number of new cases that one case generates, on average, on a completely susceptible population. It represents a measure of the effectiveness of the infection.

There is not an exact analytical solution of the SIR model \eqref{model2} in terms of the parameter $t$, however, it is possible to obtain a parametric solution in terms of a new parameter $u$. Following the results obtained in~\cite{HARKO2014184}, a parametric solution for the model~\eqref{model2} can be written as

\begin{equation}
    \begin{split}
    S(u)&=x_0u\\
    I(u)&= \frac{N}{\mathcal{R}_0  }\ln{u}-x_0u+N\\
    R(u)&= -\frac{N}{\mathcal{R}_0 }\ln{u},\\
    \end{split}
    \label{solution2} 
\end{equation}

where $\displaystyle u=e^{ -\frac{\mathcal{R}_0}{N}R}$, and $x_0=S(0)e^{\frac{\mathcal{R}_0}{N}R(0)}$. 

\begin{figure}[!ht]
    \center
    \includegraphics[width=\textwidth]{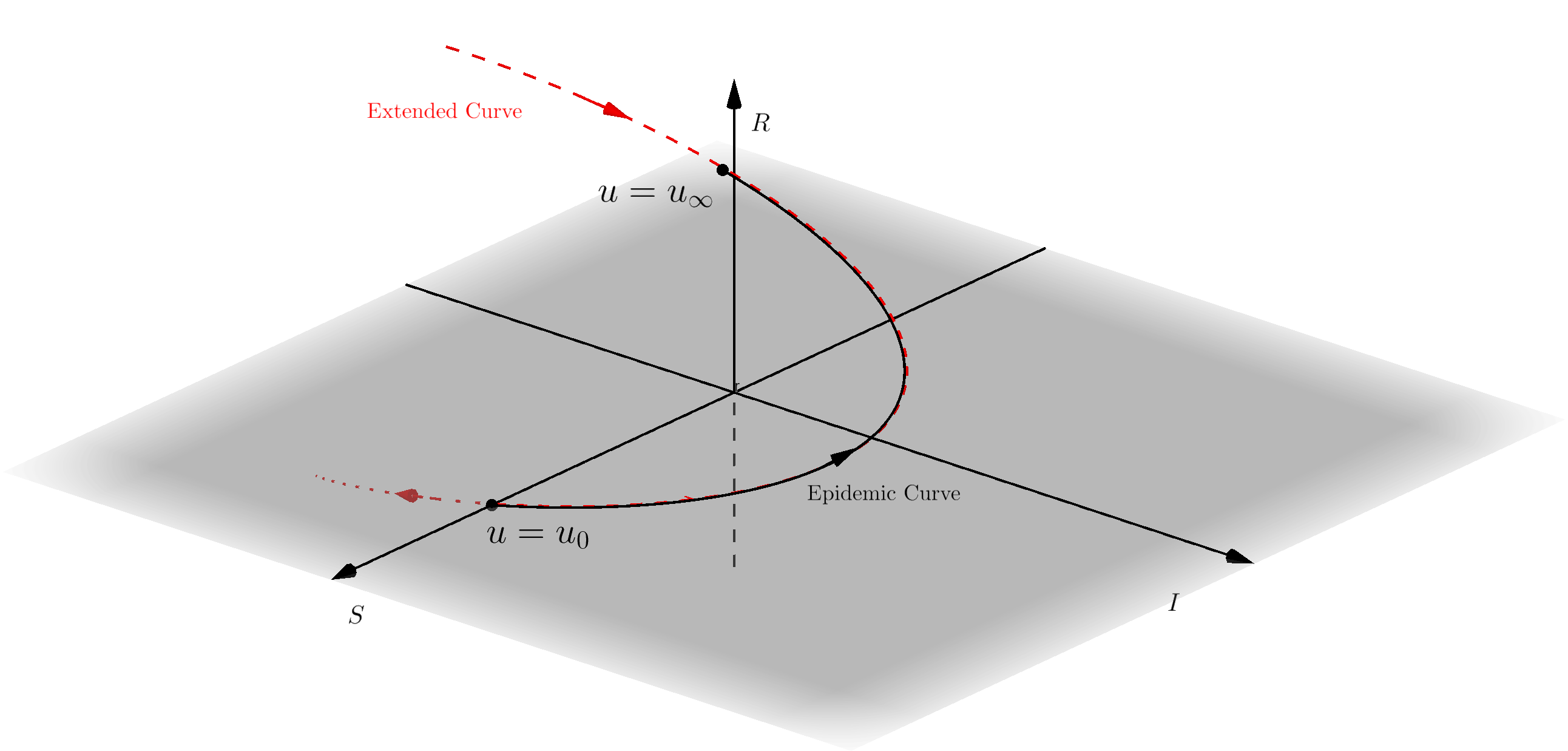}
    \caption{Parameter equations~\eqref{solution2} describe a curve (red) which contains the solution curve of the epidemiological SIR model (black) crossed in the time-opposite direction.}
    \label{Figure2}
\end{figure}
The parametric equations in~\eqref{solution2} describe a curve in $\mathbb R^3$, which, for some specific values of $u$, corresponds precisely to the solution curve of~\eqref{model2} (for details see~\cite{HARKO2014184}). In particular, to properly describe the evolution of the removed population $R$, the parameter $u$ should varies in such a way that $R$ goes from $R(0)$ to $R_\infty :=\lim_{t\to \infty}R(t)$, hence, from third equation in~\eqref{solution2}, we must have that $u$ varies from $u_0=e^{ -\frac{\mathcal{R}_0}{N}R(0)}$ to $u_\infty=e^{ -\frac{\mathcal{R}_0}{N}R_\infty}$.

Note that $u$ was defined as a decreasing function of $R$ and note also that $R$ is an increasing function of $t$. Thus, the parameter $u$ must decrease to describe the solution curve of~\eqref{model2} following the \textit{time-forward direction}.  In fact, as $0\leq R(0)\leq R_\infty\leq N$, we have that $e^{-\mathcal{R}_0}\leq u_\infty \leq u_0\leq 1$.

If one allow $u$ to varies on the interval $(0,\infty)$, then equations~\eqref{solution2} describe an extended curve which also contains the solution curve of model~\eqref{model1}, but crossed in the time-opposite direction. This occurs precisely when $u$ varies in the interval $[u_\infty, u_0] \subset (0,\infty)$  (Figure~\ref{Figure2}).

\section{Controlling the maximum number of infected} 
\label{sec.Imax}

Let $M$ be a positive constant such that $M<N$. We will consider $M$ as an epidemiological threshold which, ideally, must not be exceeded by the active infected population. In this section, we aim to obtain conditions to ensure that $I_{\max}\leq M$, where $I_{\max}$ denote the maximum value of the infected population $I$.

From now on, we consider  that $S(0)>0$ and $I(0)>0$ so $I(t)\not =0$ for all $t>0$. The following lemma establish a simple sufficient condition on the basic reproduction number $\mathcal{R}_0$ to ensure that $I_{\max}\leq M$. \\

\begin{lemma} \label{lemma1}  
If $\mathcal{R}_0\leq \frac{N}{S(0)}$ and  $I(0)\leq M$ or if  $\frac{N}{S(0)}<\mathcal{R}_0 \leq \frac{N}{N-M}$ then $I_{\max}\leq M$.
\end{lemma}

\begin{proof} 
Note first that, because we are considering $I\not=0$, second equation in~\eqref{model2} implies that $I$ is increasing if and only if $S>\frac{N}{\mathcal{R}_0}$, decreasing if and only if $S<\frac{N}{\mathcal{R}_0}$, and $\frac{dI}{dt}=0$ if and only if $S=\frac{N}{\mathcal{R}_0}$. Note also from equations~\eqref{model2}, that  $S$ and $R$ are non-increasing  and non-decreasing functions of $t$, respectively. 

Hence, if $\mathcal{R}_0\leq \frac{N}{S(0)}$ and $I(0)\leq M$, then  $S(0)\leq\frac{N}{\mathcal{R}_0}$ and $I$ is not increasing for all $t\geq 0$ and therefore $I_{\max}$ is already attained at $I(0)$ so $I_{\max}=I(0)\leq M$. 

On the other hand, if $\frac{N}{S(0)}<\mathcal{R}_0$ then $I$ is increasing until reach its maximum value $I_{\max}=I(t^*)$ for some $t^*$  satisfying $\frac{dI}{dt}(t^*)=0$ and therefore
\begin{equation}
    S(t^*)= \frac{N}{\mathcal{R}_0}.
    \label{eq.1} 
\end{equation}

From the fact that  $R(t^*)\geq 0$, and $N = S(t) + I(t) + R(t)$, follows that 
\begin{equation}
    I_{\max}=I(t^*)\leq R(t^*) + I(t^*) = N - S(t^*) = N - \frac{N}{\mathcal{R}_0} = N \left(1-\frac{1}{\mathcal{R}_0}\right).
    \label{inequality1}
\end{equation}

Finally, note that if $\mathcal{R}_0 \leq \frac{N}{N-M}$ then  $\displaystyle N \left(1-\frac{1}{\mathcal{R}_0}\right)\leq M $ and thus, from ~\eqref{inequality1} we can conclude that if $\mathcal{R}_0 \leq \frac{N}{N-M}$ then $I_{\max}\leq M$
\end{proof}

Lemma~\ref{lemma1} provides an easily verifiable condition on $\mathcal{R}_0$ to ensure that the threshold $M$ will not be exceeded. However, since it is a sufficient condition based on upper bounds on $I_{\max}$, it can be a very restrictive condition on $\mathcal{R}_0$. Furthermore, it does not provide information on the value of $I_{\max}$ if the condition is not satisfied. In the following result, we use the parametric solution~\eqref{solution2}, to obtain an expression for $I_{\max}$ that will allow us to establish a more robust condition on $\mathcal{R}_0$ to control $I_{\max}$.

\begin{proposition}\label{proposition1} 
$I_{\max}\leq \frac{N}{\mathcal{R}_0}\left(\ln\left(\frac{N}{\mathcal{R}_0S(0)}\right)-1\right)-R(0)+N$ and if $\mathcal{R}_0\geq\frac{N}{S(0)}$ the equality holds.
\end{proposition}

\begin{proof}
Lets consider the parametric equations~\eqref{solution2}. Recall that when $u$ varies on the the interval $(0,\infty)$, equations~\eqref{solution2} describe an extended curve which contains the solution curve of~\eqref{model1}, but crossed in the time-opposite direction when $u$ varies on the interval $[u_\infty, u_0] \subset (0,\infty)$. Note however that the maximum value of $I$ on this curve does not depend on the parametrization neither on the curve orientation.

From the second equation in~\eqref{solution2}, we have that

\begin{equation}
  \frac{dI}{du}=\frac{N}{\mathcal{R}_0u}-x_0 \quad\text{and}\quad \frac{d^2I}{du^2}=-\frac{N}{\mathcal{R}_0u^2}<0.
\end{equation}
  
Thus, $I$ is an strictly concave function on $u$ with a unique global maximum attained in $u^*=\frac{N}{\mathcal{R}_0x_0}$. Furthermore, $I$ is strictly increasing if $u<\frac{N}{\mathcal{R}_0x_0}$ and strictly decreasing if $u>\frac{N}{\mathcal{R}_0x_0}$.
  
The maximum possible value for $I$ along the extended curve is therefore given by

\begin{align*}
  I(u^*)&= \frac{N}{\mathcal{R}_0 }\ln \left(\frac{N}{\mathcal{R}_0x_0}\right)-\frac{x_0N}{\mathcal{R}_0x_0}+N=\frac{N}{\mathcal{R}_0 } \left(\ln\left(\frac{N}{\mathcal{R}_0x_0}\right)-1\right)+N=\frac{N}{\mathcal{R}_0}\left(\ln\left(\frac{N}{\mathcal{R}_0S(0)}\right)-1\right)-R(0)+N,
\end{align*}
  
where the last equality is obtained using that $x_0=S(0)e^{\frac{\mathcal{R}_0}{N}R(0)}$. Note that $I(u^*)$ is not necessarily equal to the maximum number of infected $I_{\max}$, but we have indeed that $I_{\max}\leq I(u^*)$. The equality will be valid if and only if the global maximum of $I$ is attained in the part of the curve corresponding to the epidemic, {\it i.e.} if and only if $u^*\in [u_\infty,u_0]$. If $\mathcal{R}_0\geq\frac{N}{S(0)}$ then $\frac{N}{\mathcal{R}_0 S(0)}\leq 1$ and thus
$$u^*=\frac{N}{\mathcal{R}_0x_0}=\frac{N}{\mathcal{R}_0S(0)}e^{-\frac{\mathcal{R}_0}{N}R(0)}\leq e^{-\frac{\mathcal{R}_0}{N}R(0)}=u_0.$$
  
If $u_*<u_\infty\leq u_0$ then $u_\infty$ and $u_0$ would be both in the decreasing side of $I$ so $I(u_\infty)\geq I(u_0)$.  but this would be absurd since $I(u_0)=I(0)>0$ and $I(u_\infty)=\lim_{t\to \infty}I(t)$ which in the case of the SIR model \eqref{model1} is equals to zero. Hence, we conclude that if $\mathcal{R}_0\geq\frac{N}{S(0)}$ then $u^*\in [u_\infty,u_0]$ and therefore in this case $I_{\max}$ and $I(u_*)$ must coincide, \textit{i.e.}
  
\begin{equation}
  I_{\max}=\frac{N}{\mathcal{R}_0 } \left(\ln\left(\frac{N}{\mathcal{R}_0x_0}\right)-1\right)+N = \frac{N}{\mathcal{R}_0}\left(\ln\left(\frac{N}{\mathcal{R}_0S(0)}\right)-1\right)-R(0)+N.
  \label{Imax1}
\end{equation}
\end{proof}

Expression~\eqref{Imax1} can be obtained also without using the parametric equations~\eqref{solution2}, but instead solving a separable ODE  obtained by dividing the first equation in~\eqref{model1} by the second one. See for example Section 2.2.7 in~\cite{Weiss2013}.

Expression~\eqref{Imax1} can be used to easily check, for some given values of the reproduction number, initial conditions and the threshold $M$, if $I_{\max}$ will exceed or not the threshold $M$. Furthermore, expression~\eqref{Imax1} can be used to estimate conditions on the parameters or on the initial conditions, ensuring $I_{\max}\leq M$. In particular, the next proposition use expression~\eqref{Imax1} and the Lambert W function to determine a necessary and sufficient condition on $\mathcal{R}_0$ to ensure that $I_{\max}\leq M$.

\begin{proposition}\label{proposition2}
Consider that $\mathcal{R}_0\geq \frac{N}{S(0)}$  and $I(0)<M<S(0)+I(0)$. Then $I_{max}\leq M$ if and only if 
\begin{equation}
    \displaystyle \mathcal{R}_0 \leq N\frac{W_{-1}\left(\frac{M-N+R(0)}{S(0)e}\right)}{M-N+R(0)},
    \label{R0max}  
\end{equation}
where $W_{-1}$ denotes the lower branch of the Lambert W function.
\end{proposition}
\begin{proof}

The condition $\mathcal{R}_0\geq \frac{N}{S(0)}$ allow us to consider the equality~\eqref{Imax1}. The inequality $I(0)<M<S(0)+I(0)$ means that the threshold $M$ has not been attained at initial conditions and also that $M$ is not impossible to be attained, because $S(0)+I(0)$ is the maximum number of population that is not removed yet and can eventually be infected at a specific time.

From second equality in~\eqref{Imax1} note that

\begin{equation}
    \frac{dI_{\max}}{d\mathcal{R}_0}=-\frac{N}{\mathcal{R}_0^2}\left(\ln\left(\frac{N}{\mathcal{R}_0S(0)}\right)-1\right)-\frac{N}{\mathcal{R}_0}\left(\frac{\mathcal{R}_0 S(0)}{N}\frac{N}{\mathcal{R}_0^2 S(0)} \right)=\frac{N}{\mathcal{R}_0^2}\ln\left(\frac{\mathcal{R}_0S(0)}{N}\right),
\end{equation}

so $I_{\max}$ is an increasing function on $\mathcal{R}_0$ when $\mathcal{R}_0>\frac{N}{S(0)}$. 

Additionally, note from~\eqref{Imax1} that if $\mathcal{R}_0 \to \infty$ then $I_{\max}\to N-R(0)=S(0)+I(0)$, so the condition  $I(0)<M<S(0)+I(0)$ implies that $M$ is in fact an attainable value for $I_{\max}$, \textit{i.e.} there exists a value $\mathcal{R}_0^*>\frac{N}{S(0)}$ such that $I_{\max}=M$. In fact, $I_{\max}\leq M$ if and only if  $\mathcal{R}_0\leq \mathcal{R}_0^*$ because we have established that $I$ is increasing on $\mathcal{R}_0$.

Using the Lambert W function, (See \ref{Append} for details), it is possible to determine the value $\mathcal{R}_0^*$ such that $I_{\max}(\mathcal{R}_0^*)=M$. From Equation~\eqref{Imax1} we have that:

\begin{align*}
    I_{\max}&=M\\
    \frac{N}{\mathcal{R}_0^*}\left(\ln\left(\frac{N}{\mathcal{R}_0^*S(0)}\right)-1\right)&=M+R(0)-N\\
    \frac{N}{\mathcal{R}_0^*S(0)}\ln\left(\frac{N}{\mathcal{R}_0^*S(0)}\right)&=\frac{N}{\mathcal{R}_0^*S(0)}+\frac{M+R(0)-N}{S(0)}\\
    v\ln v&=v+b,
\end{align*}
with $v=\frac{N}{\mathcal{R}_0^*S(0)}$ and $b=\frac{M+R(0)-N}{S(0)}$. According to Lemma~\ref{Lambert3} in \ref{Append}, the solutions to the equation $ v \ln v=v+b$  have the form $$v=\frac{b}{W(be^{-1})}$$ and because $be^{-1+1}=b=\frac{M+R(0)-N}{S(0)}$ the condition  $I(0)<M<S(0)+I(0)$ implies that $-1<b<0$. Hence, Lemma~\ref{Lambert3} also implies that $v \ln v=v+b$ have two solutions, each one corresponding to a specific branch of the function $W$. One of these solution is precisely given by $v=\frac{N}{\mathcal{R}_0^*S(0)}\leq 1$, so we must have for some branch of $W$ that
\begin{equation}
    \frac{b}{W(be^{-1})}\leq 1
\end{equation}
for  $b$ in $(-1,0)$. This is only possible for the lower branch of the Lambert W function and thus, we conclude that
$\frac{b}{W_{-1}(be^{-1})}=\frac{N}{\mathcal{R}_0^*S(0)}$ and therefore 
\begin{equation}
    \mathcal{R}_0^*=N\frac{W_{-1}\left(\frac{M-N+R(0)}{S(0)e}\right)}{M-N+R(0)}
    \label{eq.R0min}
\end{equation}
as we desired to proof.
\end{proof}

\begin{figure}[!ht]
    \center
    \includegraphics[width=0.6\textwidth]{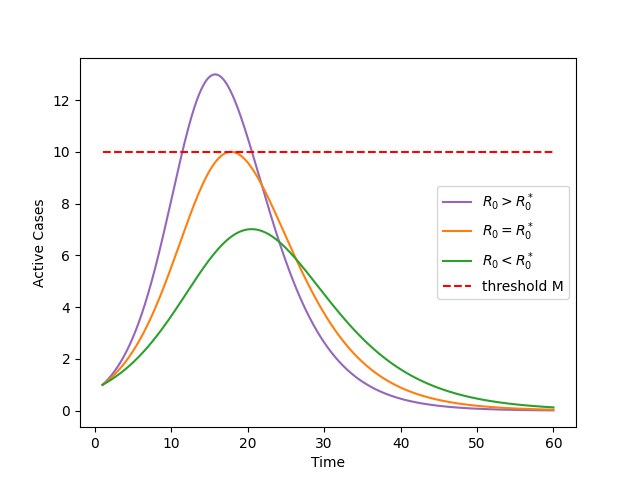}
    \caption{Infected curves for different values of $\mathcal R_0$ respect to the threshold $M=10$ and critical reproduction number $\mathcal R_0^*\approx 1.7$ as in Example \ref{example1}.}
    \label{fig_prop2}
\end{figure}

\begin{example}\label{example1}
 To illustrate the result obtained in Proposition~\ref{proposition2}, consider a scenario with total population $N=100$, recovery rate $\gamma = 1/3$, initial conditions $S(0) = 99$, $I(0) = 1$, $R(0) =0$, and an epidemiological threshold corresponding to 10\% of the total population, i.e. $M =10$. In this case, using equation \eqref{eq.R0min} and  the implementation of the Lambert W function ({\it lambertw}) from the module {\it scipy} (Scientific Python), we obtain a critical value  $\mathcal{R}_0^*\approx 1.7$. The corresponding infected curve is pictured in orange at Figure~\ref{fig_prop2}. Note that, as expected, the maximum value of active infected cases,  $I_{max}$,  corresponds exactly to the epidemiological threshold $M$.

In Figure~\ref{fig_prop2}, are also pictured the infected curves corresponding to a basic reproduction number 10\% higher (purple line) and 10\% lower (green) than the critical value $\mathcal{R}_0^*$, and as expected, the values of $I_{\max}$ are higher and lower than $M$, respectively.
\end{example}

\section{Quantifying the impact of threshold exceeding} 
\label{sec.quantifying}

In this section we propose and compare different measures to quantify the impact of a possible threshold exceeding. After Propositions~\ref{proposition1} and~\ref{proposition2}, we immediately can consider as plausible measures for threshold exceeding quantification, the following quantifiers  $Q_1$ and $Q_2$:

\begin{equation}
    Q_1 =\mathcal{R}_0 - \mathcal{R}_0^* = \mathcal{R}_0-N\frac{W_{-1}\left(\frac{M-N+R(0)}{S(0)e}\right)}{M-N+R(0)},
    \label{eq.Q1}
\end{equation}

and 

\begin{equation} 
    Q_2= I_{max} - M = \frac{N}{\mathcal{R}_0}\left(\ln\left(\frac{N}{\mathcal{R}_0S(0)}\right)-1\right)-R(0)+N-M.
    \label{eq.Q2}
\end{equation}

Quantifiers $Q_1$ and $Q_2$ are measuring how far $\mathcal{R}_0$ and $I_{\max}$ are, respectively, from the critical values for threshold $M$ exceeding.

Figures~\ref{figExQ1} and~\ref{figExQ2} illustrate quantifiers $Q_1$ and $Q_2$. The blue line in Figure~\ref{figExQ1} corresponds to the value of $I_{max}$ calculated for different values of $\mathcal{R}_0$. Note that in $\mathcal{R}_0^*$, $I_{max} = M$. The value of $Q_1$ for a specific $\mathcal{R}_0$ is the difference between $\mathcal{R}_0$ and the $\mathcal{R}_0^*$. The blue line in Figure~\ref{figExQ2} corresponds to the values of $I(t)$ and the maximum peak of this curve, $I_{max}$, is highlighted by the black dashed line. The value of $Q_2$ is the difference between $I_{max}$ and the threshold $M$.
 
\begin{figure}[ht!] 
    \label{figExQ1234}
    \centering
    \subfigure[]
    {\centering \includegraphics[width=0.45\linewidth]{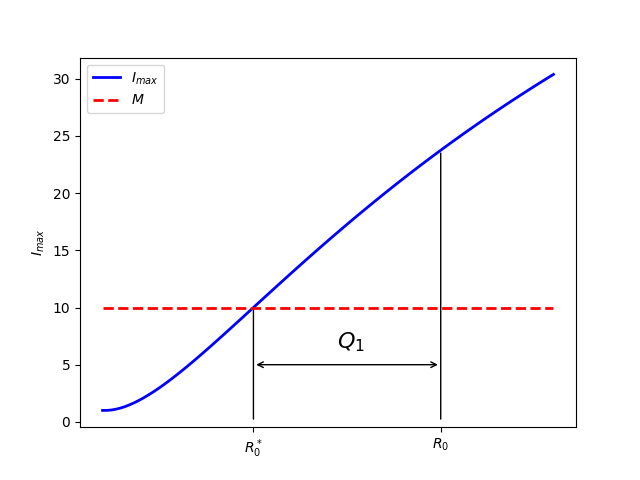}\label{figExQ1}}
    \subfigure[]
    {\centering \includegraphics[width=0.45\linewidth]{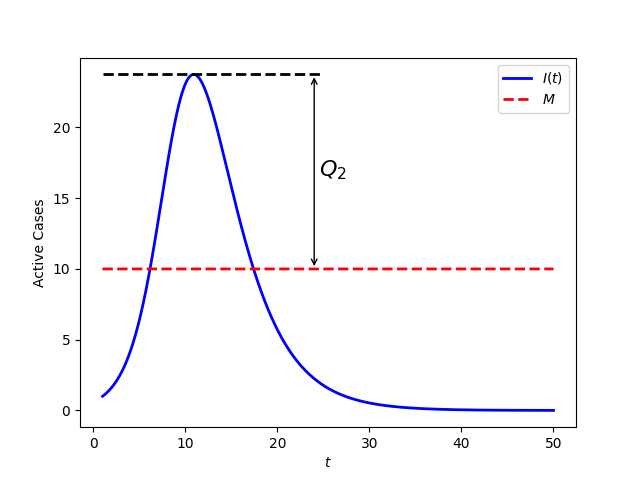}\label{figExQ2}}
    \subfigure[]
    {\centering \includegraphics[width=0.45\linewidth]{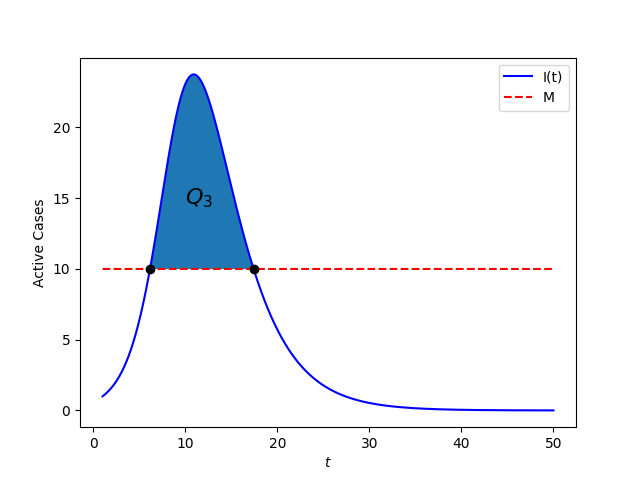}\label{figExQ3}}
    \subfigure[]
    {\centering \includegraphics[width=0.45\linewidth]{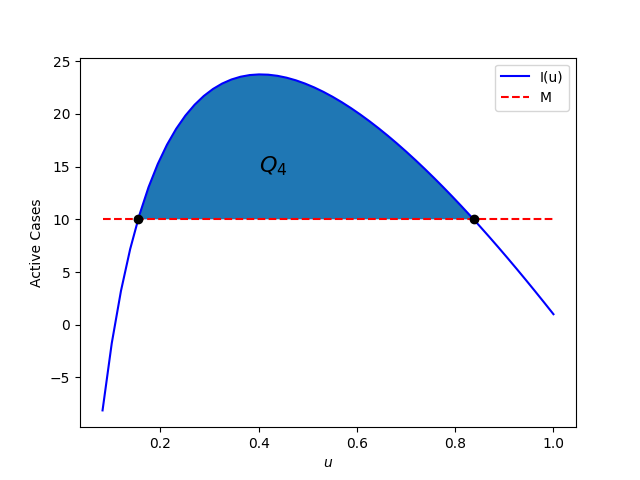}\label{figExQ4}}
	\caption{Representation of quantifiers $Q_1$, $Q_2$, $Q_3$, and $Q_4$, considering the parameters $N=100$, $\gamma = 1/3$, $\mathcal{R}_0 = 2.5$, $S(0) = N-1$, $I(0) = 1$, $R(0) =0$, and $M = 0.1\cdot N$. 
    The threshold $M$ is represented by the red dashed line.}
\end{figure}

The quantifiers $Q_1$ and $Q_2$ are only considering the impact on the epidemic peak, without explicitly consider the total impact of the threshold exceeding $M$. To take into account such concern, we can consider, as foreshadowed in the introduction (Figure~\ref{Figure1}), an integration-based measure, for example:

\begin{equation}\label{Integral1}
    Q_3=\int_{t_i}^{t_f}(I(t)-M)dt,
\end{equation}
where $t_i\leq t_f$ are values of $t$ such that $I(t_i)=M=I(t_2)$ and $M<I_{\max}$. Quantifier $Q_3$ is illustrated in Figure~\ref{figExQ3}.

The quantifier $Q_3$ is a natural expression for impact quantification, however, note that there is not an analytical expression for $I$ in terms of $t$, neither closed forms for the limits of integration $t_i$ and $t_f$, therefore $Q_3$ can only be estimated numerically.

A similar integration-based measure can be defined using the parametric equations in~\eqref{solution2}, with analogous interpretation but expressed in terms of the parameter $u$. In contrast with the previous quantifier, there exists a closed form for the integrand function in the definition of $Q_4$:

\begin{equation}
    Q_4=\int_{u_f}^{u_i}(I(u)-M)du =\int_{u_f}^{u_i} \left(\frac{N}{\mathcal{R}_0  }\ln{u}-x_0u+N-M\right)du=\left(\frac{N}{\mathcal{R}_0}(u\ln u-u)-\frac{x_0u^2}{2}+(N-M)u\right)\bigg\rvert_{u_f}^{u_i},
    \label{Integral2}
\end{equation}

where $u_f\leq u_i$ are such that $I(u_f)=M=I(u_i)$. We will see in Subsection~\ref{subsec.Q4}, that there exists closed expressions for $u_f$ and $u_i$, in terms of the Lambert W function. Thus, $Q_4$ can be expressed completely in a closed form. Quantifier $Q_4$ is illustrated in Figure~\ref{figExQ4}.
Note that in ~\ref{figExQ4} $I$ is varying according to the parameter $u$, while in
~\ref{figExQ3} $I$ is varying according to $t$. 

While $Q_3$ and $Q_4$ are defined in similar ways as defined integrals of a real function, their values are not equal in general, because are based on different parametrizations of the infected curve. It is possible to define an integration-based measure that does not depend on the parametrization, using the following line integral on the three-dimensional extended curve:

$$Q_5=\int_C f\cdot ds,$$

where $C$ corresponds to the part of the epidemic curve where $I$ is bigger than $M$ and $f$ is the scalar field defined by $f=I-M$. Quantifier $Q_5$ can be considered as a natural expression for impact quantification when considering the three-dimensional epidemic curve and the area of the surface determined by the curve an the plane $I=M$ (see Figure~\ref{Figure3}).

\begin{figure}[!ht]
    \center
    \includegraphics[width=\textwidth]{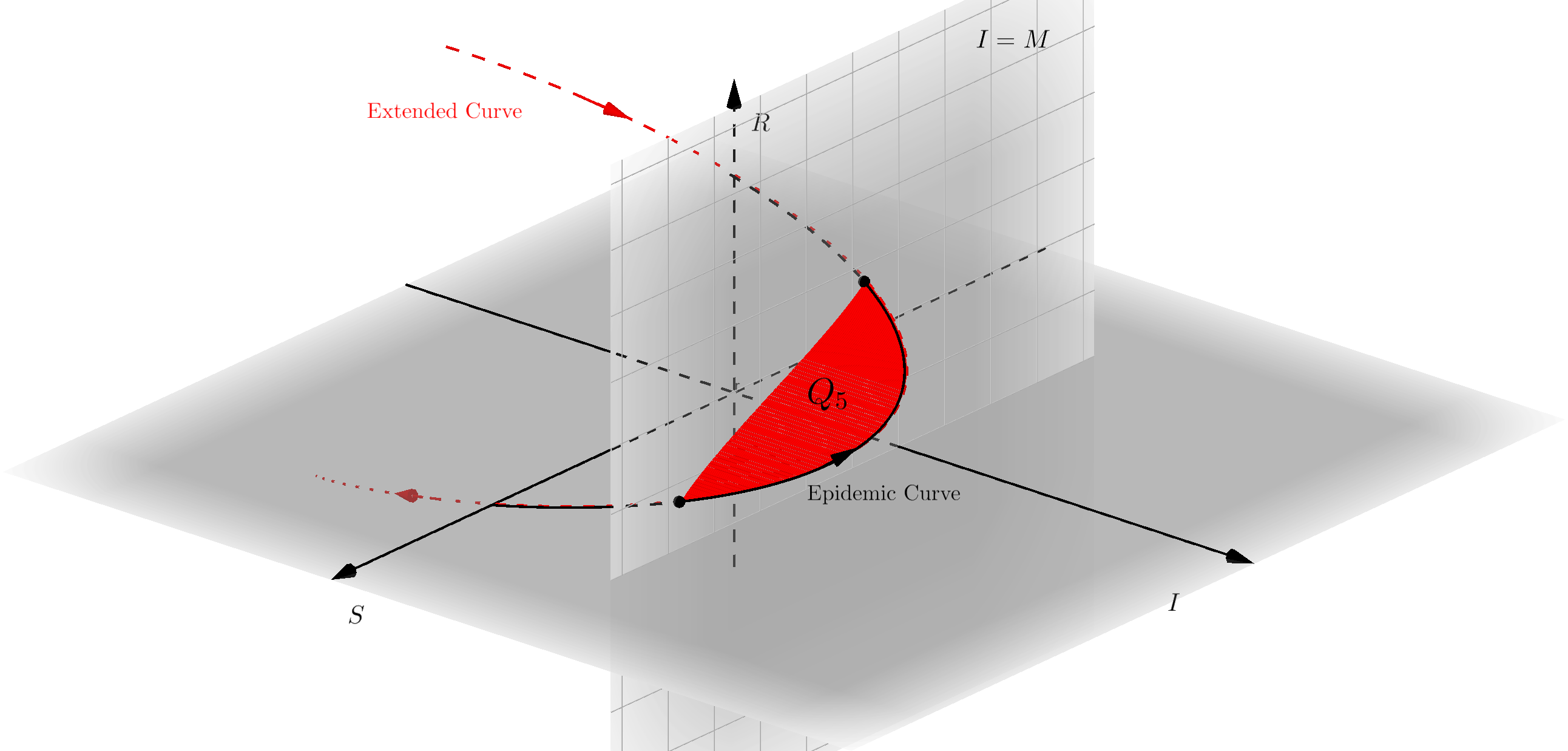}
    \caption{Line Integral $Q_5$ as a measure for threshold exceeding quantification.}
    \label{Figure3}
\end{figure}

The quantifier $Q_5$ does not depend on the parametrization or the orientation of $C$ because it is defined as the line integral of scalar fields on a smooth curve. Hence, to calculate $Q_5$ we could use any parametrization, in particular, we have that

\begin{equation}
    Q_5=\int_C f\cdot ds=\int_{t_i}^{t_f}(I(t)-M)|r_t'(t)|dt=\int_{u_f}^{u_i}(I(u)-M)|r_u'(u)|du,
\end{equation}

where $t_i$, $t_f$, $u_f$ and $u_i$ are defined as before and $r_t$ and $r_u$ are parametrizations of the epidemiological curve in terms of $t$ and $u$ respectively. As previously discussed, there are not closed expressions for  $I(t), t_i, t_f$ neither for $r_t$. However, note that for the parameter $u$, the expressions for $r_u(u)$ and $I(u)$ are given precisely by Equations~\eqref{solution2}. Thus, $Q_5$ can be expressed as

\begin{equation}
    Q_5=\int_{u_f}^{u_i}(I(u)-M)|r_u'(u)|du=\int_{u_f}^{u_i} \left(\frac{N}{\mathcal{R}_0  }\ln{u}-x_0u+N-M\right)\sqrt{2\left(x_0^2+\frac{N}{\mathcal{R}_0 u}\left(\frac{N}{\mathcal{R}_0 u}-x_0\right)\right)}du.
    \label{Integral3}
\end{equation}

In subsection~\ref{subsec.R0impact}, we will use numerical examples to compare the quantifiers introduced above and to illustrate the influence of variations on parameters $\mathcal{R}_0$ and $M$. Before that, in the next subsection, we show that $u_f$ and $u_i$ can be expressed in closed-form in terms of the Lambert W function.

\subsection{Estimation of $u_i$ and $u_f$}
\label{subsec.Q4}

The effective computation of quantifiers $Q_4$ in ~\eqref{Integral2} and $Q_5$ in ~\eqref{Integral3} requires to determine the values of $u_i$ and $u_f$ such that  $I(u_f)=M=I(u_i)$. In general, one can consider the following problem: It is possible for the extended curve to attain a given value $M$? When will this happen? What are the corresponding values of $u$? In terms of the parametric equations~\eqref{solution2}, we aim to determine solutions $u \in (0,\infty)$, for the equation
\begin{equation}
    I(u)=\frac{N}{\mathcal{R}_0  }\ln{u}-x_0u+N=M.
    \label{eq.M}
\end{equation}

We already established that the maximum possible value for $I$ is given by $\frac{N}{\mathcal{R}_0 } \left(\ln\left(\frac{N}{\mathcal{R}_0x_0}\right)-1\right)+N$, so~\eqref{eq.M} can not have solutions if $M$ is bigger than this quantity. The following proposition, which uses Lemma~\ref{Lambert2} in \ref{Append}, can be used to reach the same conclusion, but most importantly, allow us to obtain a closed-form for the solutions of~\eqref{eq.M}, and will allow us to express $u_f$ and $u_i$, in terms of the Lambert W function.

\begin{proposition}\label{ValuesI}
Let $M$ be a positive real number, $M\leq N$ and consider $S(0)>0$. 
\begin{itemize}
    \item If $M<\frac{N}{\mathcal{R}_0  } \left(\ln\left(\frac{N}{\mathcal{R}_0x_0}\right)-1\right)+N,$ then the curve $I$ attain the value $M$ in  two values of $u$ considering the two branches of $W$ in the expression
    
    \begin{equation}
        u=\frac{W\left(-x_0\frac{\mathcal{R}_0}{N} e^{\frac{\mathcal{R}_0}{N}(M-N)}\right)}{-x_0 \frac{\mathcal{R}_0}{N}}=\frac{ W\left(-\frac{\mathcal{R}_0}{N}S(0) e^{\frac{\mathcal{R}_0}{N}(R(0)+M-N)}\right)}{-\frac{\mathcal{R}_0}{N} S(0)e^{\frac{\mathcal{R}_0}{N}R(0)}}.
        \label{ucritical}
    \end{equation}
        
    \item If $\frac{N}{\mathcal{R}_0  } \left(\ln\left(\frac{N}{\mathcal{R}_0x_0}\right)-1\right)+N<M$, the extended infected curve never attains the value $M$.
\end{itemize}
\end{proposition}

\begin{proof}
From Equation~\eqref{eq.M} we have that

\begin{align*}
    \frac{N}{\mathcal{R}_0  }\ln{u}-x_0u+N&=M\\
    \ln{u}&=\frac{\mathcal{R}_0}{N} x_0 u+\frac{\mathcal{R}_0}{N}\left(M-N\right),
\end{align*}
 
so the expression~\eqref{ucritical} for the solutions $u$, when they exist, can be obtained from Lemma~\ref{Lambert2} considering $a=\frac{\mathcal{R}_0 x_0}{N}$ and $b=\frac{\mathcal{R}_0}{N}\left(M-N\right)$.

Let us analyze then, if under the conditions considered, Equation~\eqref{ucritical} has two solutions. The condition for the existence of two solutions is given in Lemma~\ref{Lambert2} by $0<ae^{b+1}<1$ which in this case is equivalent to

$$0<\frac{\mathcal{R}_0 x_0}{N}e^{\frac{\mathcal{R}_0}{N}\left(M-N\right)+1}<1.$$

Left side inequality is immediately satisfied because $x_0>0$ when $S(0)>0$ and the right side is equivalent to

\begin{align*}
    \frac{\mathcal{R}_0 x_0}{N}e^{\frac{\mathcal{R}_0}{N}\left(M-N\right)+1}&<1\\
    e^{\frac{\mathcal{R}_0}{N}\left(M-N\right)+1}&<\frac{N}{\mathcal{R}_0 x_0}\\
    \frac{\mathcal{R}_0}{N}\left(M-N\right)+1&<\ln\left(\frac{N}{\mathcal{R}_0x_0}\right)\\
    M&<\frac{N}{\mathcal{R}_0  } \left(\ln\left(\frac{N}{\mathcal{R}_0x_0}\right)-1\right)+N,
\end{align*}
which is precisely our hypothesis. By the same reasoning, the inequality $\frac{N}{\mathcal{R}_0} \left(\ln\left(\frac{N}{\mathcal{R}_0x_0}\right)-1\right)+N<M$ is equivalent to $1<ae^{b+1}$ which implies by Lemma~\ref{Lambert2} that in this case Equation \eqref{eq.M} has no solutions.\end{proof}

\subsection{Numerical Examples} 
\label{subsec.R0impact}
Although the quantifiers presented above aim to measure the same phenomenon (the impact of exceeding the threshold $M$), each one does it in a different way. To illustrate these differences, in this section we present some numerical comparisons between the quantifiers.

The values of $Q_1$ and $Q_2$ were calculated with a closed expressions (Equations~\eqref{eq.Q1} and~\eqref{eq.Q2}, respectively). 
The value of $Q_4$ was also calculated by a closed expression (right side of~\eqref{Integral2}), where the integration extremes $u_i$ and $u_f$ were calculated using Equation~\eqref{ucritical} with the implementation of the Lambert W function ({\it lambertw}) from the module {\it scipy} (Scientific Python). 
On the other hand, the values of $Q_3$ and $Q_5$ were obtained using numerical approximations.
For $Q_3$, the Equation~\eqref{Integral1} was used where the values of $I(t)$ were calculated via a numerical solution of EDO's (using function {\it odeint} from {\it scipy}); the integration extremes were calculated numerically by performing a search among the discretized values of $I (t)$; and the integral was also performed numerically (using function {\it trapz} from the {\it Numpy} library for the Python).
For $Q_5$ was used the right side of Equation~\eqref{Integral3}. 
As for $Q_3$, numerical integration was used. 
However, the integration extremes were calculated exactly by Equation~\eqref{ucritical} and $Q_5$ does not need the numerical solution of the ODE system.

\begin{figure}[ht!] 
    \centering
    \subfigure[$Q_1$]
    {\centering \includegraphics[width=0.45\linewidth]{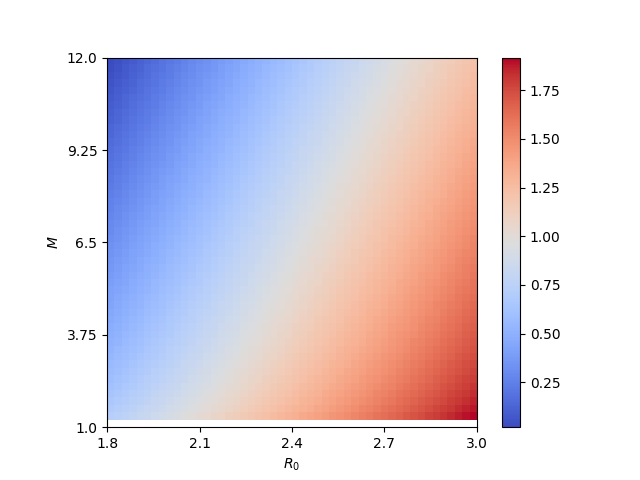}\label{fig_hmQ1}}
    \subfigure[$Q_2$]
    {\centering \includegraphics[width=0.45\linewidth]{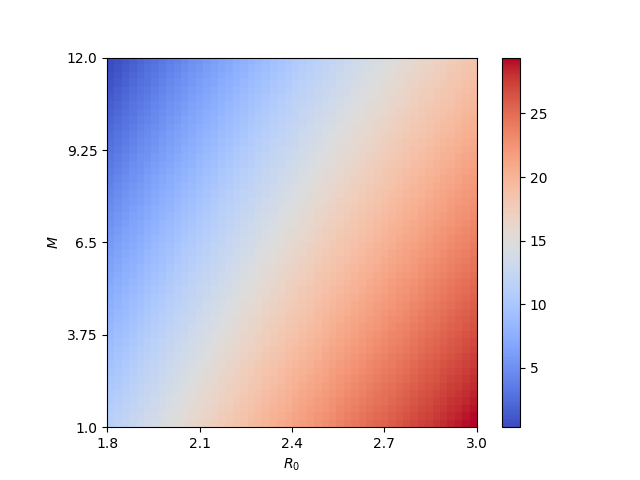}\label{fig_hmQ2}}
    \subfigure[$Q_3$]
    {\centering \includegraphics[width=0.45\linewidth]{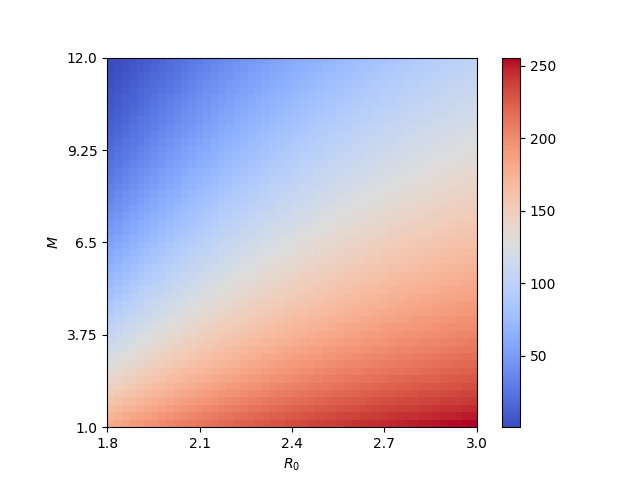}\label{fig_hmQ3}}
    \subfigure[$Q_4$]
    {\centering \includegraphics[width=0.45\linewidth]{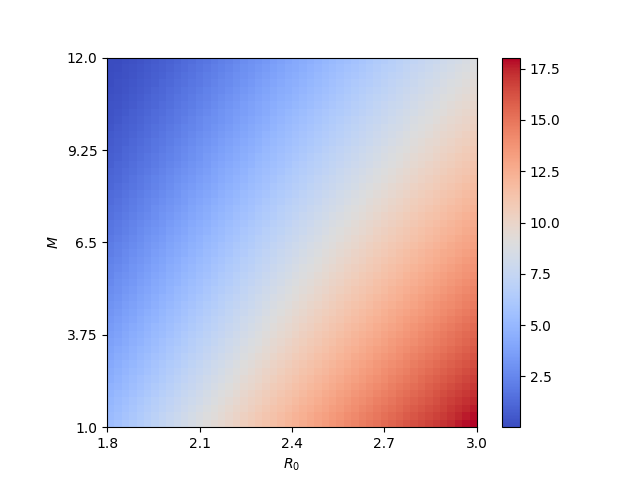}\label{fig_hmQ4}}
    \subfigure[$Q_5$]
    {\centering \includegraphics[width=0.45\linewidth]{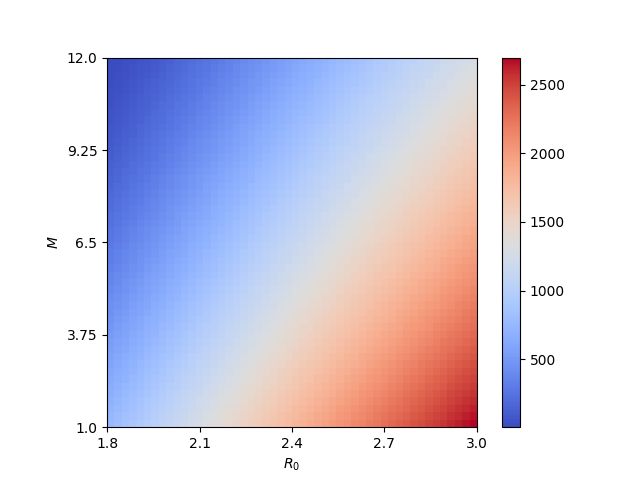}\label{fig_hmQ5}}\caption{
	Heat maps for the values of each quantifier for different values of $\mathcal{R}_0$ and $M$. The parameters considered for these simulations were $N=100$, $\gamma = 1/3$, $S(0) = N-1$, $I(0) = 1$, $R(0) =0$, $M \in [1,12]$ and $\mathcal{R}_0 \in [1.8, 3]$.} 
	\label{fig_hmQ}
\end{figure}

\begin{figure}[ht!] 
    \centering
    \subfigure[$Q_1$]
    {\centering \includegraphics[width=0.45\linewidth]{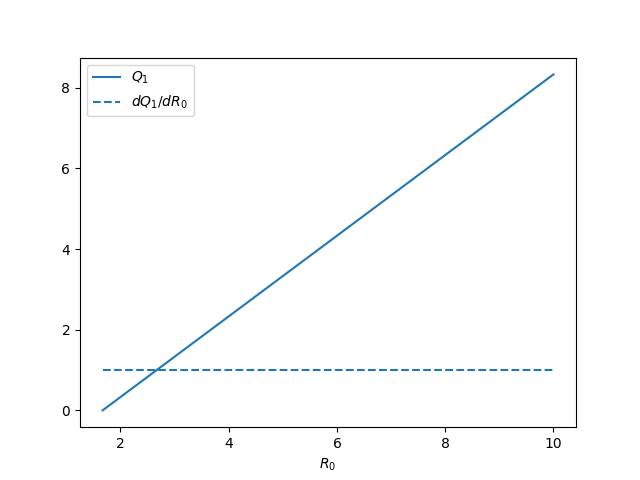}\label{fig_dQ1}}
    \subfigure[$Q_2$]
    {\centering \includegraphics[width=0.45\linewidth]{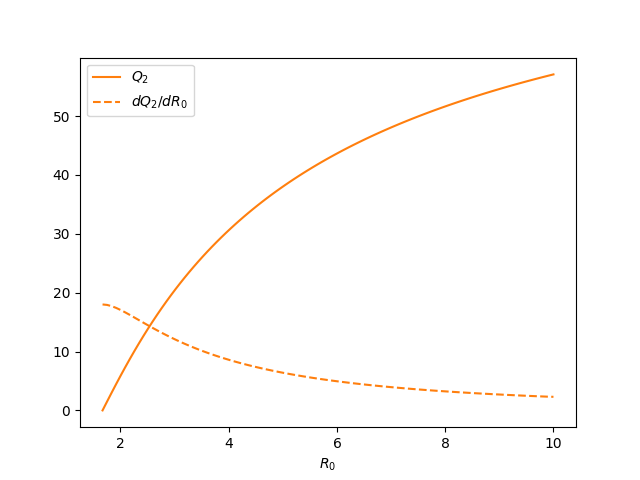}\label{fig_dQ2}}
    \subfigure[$Q_3$]
    {\centering \includegraphics[width=0.45\linewidth]{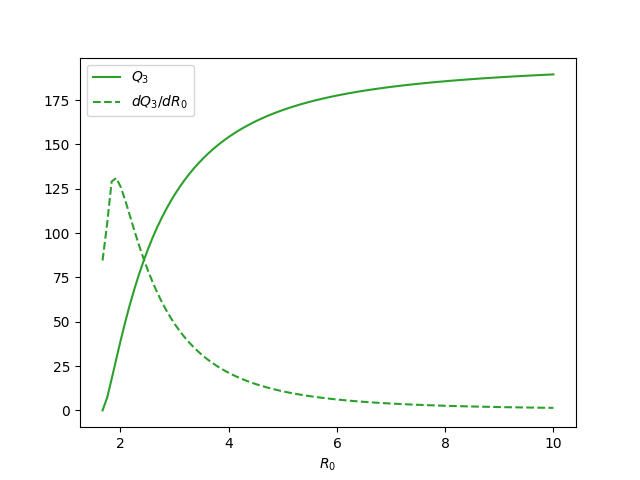}\label{fig_dQ3}}
    \subfigure[$Q_4$]
    {\centering \includegraphics[width=0.45\linewidth]{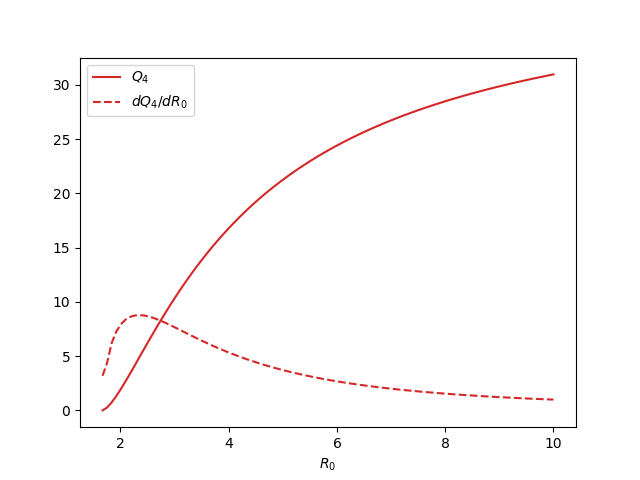}\label{fig_dQ4}}
    \subfigure[$Q_5$]
    {\centering \includegraphics[width=0.45\linewidth]{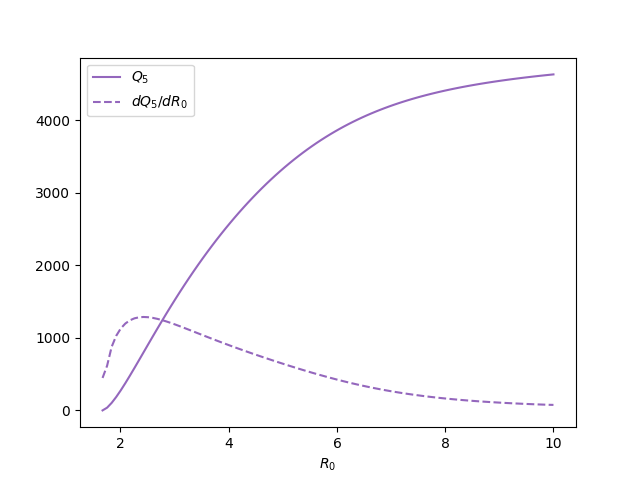}\label{fig_dQ5}}
	\caption{
    Values of the quantifiers $Q_1$ (blue), $Q_2$ (orange), $Q_3$ (green), $Q_4$ (red), and $Q_5$ (purple), in solid line, and their derivatives with respect to $\mathcal{R}_0$, in dashed line.
	The parameters considered were $N=100$, $\gamma = 1/3$, $S(0) = N-1$, $I(0) = 1$, $R(0) =0$, $M = 0.1\cdot N$ and $\mathcal{R}_0 \in [\mathcal{R}_0^*, \overline{\mathcal{R}_0}]$ where $\mathcal{R}_0^*$ is given by Equation~\eqref{eq.R0min} and  $\overline{\mathcal{R}_0} = 10$.}
	\label{fig_derivadasQ}
\end{figure}

Note that the basic reproduction number $\mathcal{R}_0$ and the epidemiological threshold $M$ play a fundamental role in the dynamics and, consequently, on the quantifiers. 
In order to illustrate the behaviors of the quantifiers with respect to these parameters, the Figure~\ref{fig_hmQ} shows heat maps indicating the intensity of each quantifier for different values of $\mathcal{R}_0$ and $M$.

At first glance, quantifiers behave similarly with respect to variations on $M$ and $\mathcal{R}_0$. As expected, we see low intensity for small $\mathcal{R}_0$ and large $M$, and a gradual increase on intensity as $\mathcal{R}_0$ increases or $M$ decreases. Maximum intensity is attained for the highest values of $\mathcal{R}_0$ and the lowest values of $M$. Note however the difference on scales for e each quantifier. 

A closer look at each quantifier will allow us to observe slightly different behaviors. Note for example that even for a low fixed $M$ (such as $ M = 3.75 $), $Q_3$ has a very high intensity even at low $\mathcal{R}_0$ values, while $Q_5$ will only considerably increase intensity for higher values of $\mathcal{R}_0$. 

 Figure~\ref{fig_derivadasQ} shows the quantifiers as functions of $\mathcal{R}_0$, and also presents their derivatives (in relation to $\mathcal{R}_0$).

With the exception of $Q_1$ which has a constant derivative, all other quantifiers loss sensitivity to changes (derivative goes to zero) for large values of $\mathcal R_0$. This sensitivity loss starts at different values of $\mathcal R_0$ for each quantifier and evolve also in different ways.

As all quantifiers have different value scales, to compare them simultaneously on a specific interval, we can normalize its values with respect to the maximum value, which is attained at the right extreme of the interval, since all quantifiers are increasing with respect to $\mathcal{R}_0$. To illustrate this kind of comparison, Figure~\ref{fig_Qnormalizado} presents the normalized values for all quantifiers and Figure~\ref{fig_QlogD} shows the logarithmic derivatives of the quantifiers with respect to $\mathcal{R}_0$, that is, the quotient between $\frac{dQ_i}{d\mathcal{R}_0}$ and $Q_i$, for $i=1,2,3,4$ or $5$. In this case, we can establish for example from Figure ~\ref{fig_Qnormalizado} that $Q_3$ approach its maximum value comparatively faster than the other quantifiers, and from Figure \ref{fig_QlogD} that $Q_5$ is slightly more sensitive for small values of $\mathcal R_0$ than the other quantifiers.

\begin{figure}[ht!] 
    \label{fig_Qagrupados}
    \centering
    \subfigure[Normalized Quantifiers]
    {\centering \includegraphics[width=0.45\linewidth]{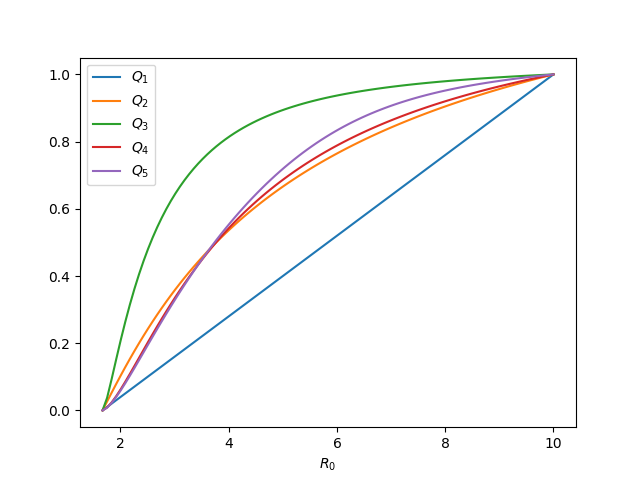}\label{fig_Qnormalizado}}
    \subfigure[Logarithmic derivatives]
    {\centering \includegraphics[width=0.45\linewidth]{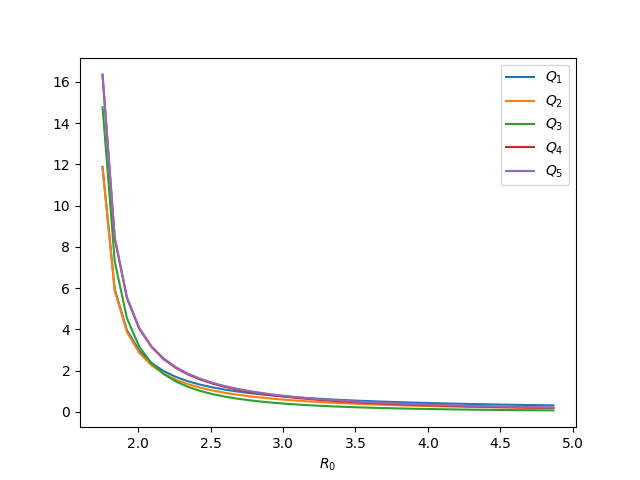}\label{fig_QlogD}}
	\caption{
	Comparison between quantifiers. 
    Figure (a) shows the normalized quantifiers as a function of $\mathcal{R}_0$. 
    Figure (b) shows logarithmic derivatives of the quantifiers, with respect to $\mathcal{R}_0$. 
	The parameters considered were $N=100$, $\gamma = 1/3$, $S(0) = N-1$, $I(0) = 1$, $R(0)=0$, $M=0.1\cdot N$,  $\mathcal{R}_0 \in [\mathcal{R}_0^*, 10]$ in (a) and $\mathcal{R}_0 \in [\mathcal{R}_0^*, 5]$  in (b) where $\mathcal{R}_0^*\approx 1.7$. is obtained by using Equation~\eqref{eq.R0min}.}
\end{figure}

\section{Final Comments} 
\label{sec.finalcomments}

The main contribution of this paper is the determination of sufficient and necessary conditions to ensure the infected curve remains below a given threshold (Proposition \ref{proposition2}) and the quantification of the threshold exceeding (Quantifiers $Q_i, i=1,\dots,5$) when this occurs. Some of the proposed quantifiers are measuring the impact explicitly on the epidemic peak, (quantifiers $Q_1$ and $Q_2$), while the others aim to achieve a global impact quantification, using some type of integration; being with respect to the parameter $t$ (Quantifier $Q_3$),  with respect to the parameter $u$ (Quantifier $Q_4$) or even using a line integral which does not depend on the curve parametrization (Quantifier $Q_5$). Future work should extend these results to more realistic epidemic models.

\appendix
\section{Lambert W function}
\label{Append}

The Lambert $W$ function is a multi-valued function corresponding to the inverse relation of the function $f(x) = xe^x$, so $W(x)$ gives solutions $u$ to the equation $ue^u=x$ {\it i.e.}  $W(x)$ satisfies that 

$$ W(x)e^{W(x)}=x.$$

For the real case, the Lambert $W$ function is defined only for $-\frac{1}{e}\leq x$. This and other properties related are summarized in the following lemma, enunciated here without proof. For proof details and additional properties of the Lambert $W$ function see~\cite{Corless1996}.

\begin{lemma}\label{Lambert1} 
For $x\in \mathbb R$, the Lambert $W$ function on $x$ is defined only for $-\frac{1}{e}\leq x$. If $-\frac{1}{e}< x <0$, then the equation $ye^y=x$ has two solution and therefore $W(x)$ has two possible values denoted by $W_{-1}(x)$ and $W_0(x)$. If  $0\leq x$ the equation has a unique solution denoted by $y=W_{0}(x)$.  In the point $x=-\frac{1}{e}$ the equation has a unique solution $W\left(-\frac{1}{e}\right)=-1$.
\end{lemma}

The Lambert $W$ function can be used to solve equations involving natural logarithms as described in the following results.

\begin{lemma}\label{Lambert2} 
Let $a,b$ be real numbers with $a\not=0$. Consider the equation 

\begin{equation}
    \ln{u}=au+b,
    \label{equationU}
\end{equation}

for $u>0$. If there exists solutions, they can be expressed in terms of the Lambert $W$ function as 
\begin{equation}
    u=\frac{W(-ae^{b})}{-a}
\end{equation}
The Equation~\eqref{equationU} has no solutions if $1<ae^{b+1}$; it has two solution if $0<ae^{b+1}<1$ and has a unique solution if $ae^{b+1}=1$ or $ae^{b}<0$. 
\end{lemma}

\begin{proof} 
From Equation~\eqref{equationU} we have that

\begin{align*}
    \ln{u}&=au+b\\
    u&= e^b e^{au}\\
    -au e^{-au}&= -ae^b\\
    -a&u=W(-ae^b)\\
     u&=\frac{W(-ae^{b})}{-a}.
\end{align*}
        
Note that according to Lemma~\ref{Lambert1}, $W(-ae^{b})$ is not defined if $-ae^{b}<-\frac{1}{e}$ which is equivalent to $1<ae^{b+1}$, so the first affirmation is established. The rest of the conditions follows in a similar way from Lemma~\ref{Lambert1}. \end{proof}

\begin{lemma}\label{Lambert3} 
Let $a,b$ be real numbers with $b\not=0$. Consider the equation 
\begin{equation}
    v\ln{v}=av+b,
    \label{equationU2}
\end{equation}
for $v>0$. If there exists solutions, they can be expressed in terms of the Lambert $W$ function as 
\begin{equation}
    v=\frac{b}{W(be^{-a})}.
\end{equation}
The equation \eqref{equationU} has no solutions if $be^{-a+1}<-1$; it has two solution if $-1<be^{-a+1}<0$ and has a unique solution if $be^{-a+1}=-1$ or $0<be^{-a}$.
\end{lemma}

\begin{proof}
From Equation~\eqref{equationU2} we have that

\begin{align*}
    v\ln{v}&=av+b\\
    -(-\ln v)&=a +\frac{b}{v}\\
    \ln\left(\frac{1}{v}\right)&=-a-\frac{b}{v}\\
    \ln u&=-bu-a,\\
    \end{align*}
with $u=\frac{1}{v}$ and the desired result follows from Lemma~\ref{Lambert2}.   \end{proof}
  


\bibliography{references}  

\begin{thebibliography}{10}
\expandafter\ifx\csname url\endcsname\relax
  \def\url#1{\texttt{#1}}\fi
\expandafter\ifx\csname urlprefix\endcsname\relax\def\urlprefix{URL }\fi
\expandafter\ifx\csname href\endcsname\relax
  \def\href#1#2{#2} \def\path#1{#1}\fi

\bibitem{TWP}
T.~W. Post, Why outbreaks like coronavirus spread exponentially, and how to
  \textit{flatten the curve},
  https://www.washingtonpost.com/graphics/2020/world/corona-simulator/, Last
  accessed on August 15, 2020 (2020).

\bibitem{CBS}
CBS, \textit{Flattening the curve}: Why we need to cancel everything and stay
  home to help stop coronavirus,
  https://www.cbsnews.com/news/flattening-the-curve-coronavirus-graph-social-distancing-self-quarantine-no-large-events-covid-19/,
  Last accessed on August 15, 2020 (2020).

\bibitem{DMUK}
D.~Mail, Flattening the curve: Charts reveal how restricting people's movements
  can stop the coronavirus pandemic from overwhelming the nhs,
  https://www.dailymail.co.uk/news/article-8104003/Shocking-graphs-reveal-coronavirus-crisis-escalating-outside-China.html/,
  Last accessed on August 15, 2020 (2020).

\bibitem{HARKO2014184}
T.~Harko, F.~S. Lobo, M.~Mak, Exact analytical solutions of the
  susceptible-infected-recovered (sir) epidemic model and of the sir model with
  equal death and birth rates, Applied Mathematics and Computation 236 (2014)
  184 -- 194.

\bibitem{Corless1996}
R.~Corless, G.~Gonnet, D.~Hare, D.~Jeffrey, D.~Knuth, On the lambert w
  function, Advances in Computational Mathematics 5 (1996) 329 -- 359.

\bibitem{RELUGA2004249}
T.~Reluga, A two-phase epidemic driven by diffusion, Journal of Theoretical
  Biology 229~(2) (2004) 249 -- 261.

\bibitem{Wang2010}
F.~Wang, Application of the lambert w function to the sir epidemic model, The
  College Mathematics Journal 41~(2) (2010) 156--159.

\bibitem{Pakes2015}
A.~G. {Pakes}, Lambert's w meets kermack–mckendrick epidemics, IMA Journal of
  Applied Mathematics 80~(5) (2015) 1368--1386.

\bibitem{Xiao2013}
Y.~Xiao, T.~Zhao, S.~Tang, Dynamics of an infectious diseases with
  media/psychology induced non-smooth incidence, Mathematical biosciences and
  engineering : MBE 10~(2) (2013) 445—461.

\bibitem{Lehtonen2016}
J.~Lehtonen, The lambert w function in ecological and evolutionary models,
  Methods in Ecology and Evolution 7~(9) (2016) 1110--1118.

\bibitem{martcheva2015introduction}
M.~Martcheva, An introduction to mathematical epidemiology, Vol.~61, Springer,
  2015.

\bibitem{Weiss2013}
H.~Weiss, The sir model and the foundations of public health, Materials
  Matemàtics 2013 (2013).

\end{thebibliography}

\end{document}